\documentclass[a4paper,12pt]{article}
    \usepackage[top=2cm,bottom=2cm,left=2.5cm,right=2.5cm]{geometry}
    \usepackage{cite, amsmath, amssymb}
    \usepackage[margin=0.5cm,%
                font=small,%
                format=hang,%
                labelsep=period,%
                labelfont=bf]{caption}
\usepackage{amsmath}
\usepackage{amsfonts}
\usepackage{amssymb}
\usepackage{cite}
\usepackage{amsthm}
\usepackage{makeidx}
\usepackage{graphicx}
\usepackage{mathrsfs,xcolor,multicol}
\usepackage{enumerate}
\usepackage{blkarray}
\usepackage{bm}
\usepackage{setspace}
\usepackage{float}
\usepackage{authblk}
\usepackage{multirow}
\usepackage{booktabs}
\usepackage[ruled,vlined]{algorithm2e}
\usepackage{blkarray}
\usepackage{authblk}
\usepackage{multirow}
\usepackage{lineno}

\usepackage{blkarray, bigstrut}
\usepackage{authblk}
\usepackage{multirow}

\usepackage{fancyhdr}
\usepackage{lastpage}

\usepackage{amsmath, amssymb, booktabs, lastpage,}
\usepackage{multicol}
\usepackage{hyperref}
\usepackage{blkarray}
\usepackage{empheq} 

\numberwithin{table}{section}
\numberwithin{equation}{section}

\theoremstyle{plain}
\newtheorem{theorem}{Theorem}[section]

\newtheorem{definition}[theorem]{Definition}

\newtheorem{example}[theorem]{Example}

\newtheorem{corollary}[theorem]{Corollary}
\newtheorem{remark}[theorem]{Remark}

\usepackage{tikz}
\usetikzlibrary{calc, arrows}

\newcommand{\NN}{\mathcal{N}} 
\newcommand{\CC}{\mathcal{C}} 
\newcommand{\RR}{\mathcal{R}} 

\newcommand{\KK}{\mathcal{K}} 

\usepackage{authblk}
\author[1]{ \textbf{Al Jay Lan J. Alamin}}
\author[1,*]{ \textbf{Bryan S. Hernandez}}

\affil[1]{\small \textit{Institute of Mathematics, University of the Philippines Diliman, Quezon City 1101, Philippines}}
\affil[*]{Email address: \texttt{bshernandez@up.edu.ph}}

\title{\textbf{Positive steady states of a class of power law systems with independent decompositions}}
\date{}

\begin{document}
\maketitle
\begin{abstract} 
Power law systems have been studied extensively due to their wide-ranging applications, particularly in chemistry. In this work, we focus on power law systems that can be decomposed into stoichiometrically independent subsystems. We show that for such systems where the ranks of the augmented matrices containing the kinetic order vectors of the underlying subnetworks sum up to the rank of the augmented matrix containing the kinetic order vectors of the entire network, then the existence of the positive steady states of each stoichiometrically independent subsystem is a necessary and sufficient condition for the existence of the positive steady states of the given power law system. We demonstrate the result through illustrative examples. One of which is a network of a carbon cycle model that satisfies the assumption, while the other network fails to meet the assumption. Finally, using the aforementioned result, we present a systematic method for deriving positive steady state parametrizations for the mentioned subclass of power law systems, which is a generalization of our recent method for mass action systems.
	
\end{abstract}

\section{Introduction}

Throughout the past decade, network-based approaches to \emph{power law systems} have been widely used to study carbon cycle models \cite{FOMF2021,FOME2021,FOME2023,fortun2018deficiency,hernandez2023network,HEME2022jomc,HEME2023jomc,HEMEDE2020jomc} and other important models in biochemistry \cite{AJLM2017,AJLM2015,TAME2023,TAME2018}.
{Here, we consider power law systems 
in which the underlying networks have
\emph{stoichiometrically independent} (or simply \emph{independent}) decompositions.} A \emph{decomposition} of a network is induced by partitioning its set of reactions into disjoint subsets.
A decomposition is \emph{independent} if the rank of the stoichiometric matrix of a network is equal to the sum of the ranks of the stoichiometric matrices of its individual subnetworks \cite{feinberg2019crnt,hernandez2021independent}. {Independent decompositions are ubiquitous in important chemical, biological and biochemical systems \cite{FOMF2021,FOME2023,Hernandezetal2022,HEME2022jomc,HEME2023jomc,hernandez2023framework,hernandez2021independent,LUML2021,TAME2023,Villareal2023}.
In particular, the ubiquity is evident in carbon cycle models \cite{FOMF2021,FOME2023,HEME2022jomc,HEME2023jomc}, signaling pathways, \cite{hernandez2023framework,hernandez2021independent,LUML2021,Villareal2023} and genetic circuit networks \cite{hernandez2023framework}.}
Importantly, M. Feinberg showed that for a network with independent decompositions, the set of positive steady states the network is equal to the intersection of the sets of positive steady states of its independent subnetworks \cite{Feinberg1987stability,feinberg2019crnt}.

In 2012, B. Boros \cite{boros:paper:thesis,boros:thesis} provided conditions on the non-emptiness of the set of positive steady states of mass action systems that can be decomposed into independent linkage classes. This type is specific to decomposing the underlying network into its connected components.
In 2023, Hernandez et al. \cite{hernandez2023framework} then showed that for mass action systems that can be decomposed into independent subnetworks
where the deficiencies of the underlying individual subnetworks sum up to the deficiency of the entire network, and the ranks of the augmented matrices of reactant complexes of the subnetworks sum up to the rank of the augmented matrix of reactant complexes of the entire network, the existence of positive steady states of each independent subsystem implies the existence of positive steady states of the given power law system.
In this work, we have extended this result to {power law systems} and have shown to illustrate in a carbon cycle model by Anderies et al. \cite{anderies2013topology}. Furthermore, we illustrate, by example, that the necessary and sufficient condition will generally not hold if the conditions are not met.
To check if such power law system has positive steady states, we just break the network into its independent decompositions, if any, and determine if each subnetwork has positive steady states. If one of them has no positive steady states, then the whole network also has no positive steady states, as well.

In 2019, Johnston et al. \cite{JMP2019:parametrization} established a method for solving positive steady states of generalized mass action systems, which has been useful to parametrize positive steady states of mass action systems. In 2023, Hernandez et al. \cite{hernandez2023framework} proposed to decompose the underlying CRN into independent subnetworks first before applying the method of Johnston et al. \cite{JMP2019:parametrization}. In the same year, Hernandez and Buendicho \cite{hernandez2023network} illustrated that these approaches are also useful for power law systems. Most recently, Villareal et al. \cite{Villareal2023} extended the method for networks under mass action kinetics that can be decomposed to independent and identical subnetworks, by employing the principle of mathematical induction. In this work, we extend the framework of parametrization to accommodate {power law systems} systems.
Parametrization of positive steady states in biochemical systems can be useful for determining critical properties such as absolute concentration robustness, robust perfect adaptation, and multistationarity
\cite{dickenstein2019multistationarity,hirono2023,JMP2019:parametrization,shinar2010structural}.

	
	\section{Preliminaries}
	\label{prelim}

\subsection{Chemical reaction networks}

This section provides essential background on chemical reaction networks pertinent to this paper. Importantly, we present a formal definition of a chemical reaction network.

\begin{definition}
A \emph{chemical reaction network} or \emph{CRN} 
is a triple of nonempty finite sets, i.e., $\mathcal{N}=(\mathcal{S}, \mathcal{C}, \mathcal{R})$ where
    \begin{itemize}
        \item[a.] $\mathcal{S} = \left\{X_i:i=1,\ldots,m \right\}$ is the set of \emph{species},
        \item[b.] $\mathcal{C} = \{C_i:i=1,\ldots,n\}$ is the set of \emph{complexes}, which are non-negative linear combinations of the species, and
        \item[c.] $\mathcal{R} = \{R_i:i=1,\ldots,r\} \subset \mathcal{C} \times \mathcal{C}$ is the set of \emph{reactions}.
    \end{itemize}   
\end{definition}

We typically represent a reaction $(y,y')$ as $y \to y'$, where $y$ and $y'$ are the \emph{reactant complex} and the \emph{product complex}, respectively. Moreover, the associated \emph{reaction vector} of this reaction is defined as $y' - y$, which is a linear combination of species in the network.
The linear subspace $S$ of $\mathbb{R}^m$, with $m$ as the number of species, spanned by the reaction vectors is called the \emph{stoichiometric subspace} of the network. Hence, $S = \mathrm{span}\{y' - y \in \mathbb{R}^m \mid y \rightarrow y' \in \mathcal{R}\}$.
The \emph{stoichiometric matrix} $N$ of the network is an $m \times r$ matrix, where each column corresponds to the coefficients of the associated species in the respective reaction vector. Hence, the dimension of $S$ coincides with the rank of $N$.

\begin{definition}
The {\emph{deficiency}} of a CRN is $\delta:=n-\ell-s$ where $n$ is the number of complexes, $\ell$ is the number of linkage classes (i.e., connected components), and $s:={\rm{dim \ }} S$ (which is the same as ${\rm{rank \ }} N$).
\end{definition}

\subsection{Chemical kinetic systems}

To characterize the temporal evolution of species concentrations, a CRN is equipped with kinetics, defined as follows.

\begin{definition}
A \emph{kinetics} for a reaction network $\mathcal{N}=(\mathcal{S}, \mathcal{C}, \mathcal{R})$ is an assignment to
each reaction $y \to y' \in \mathcal{R}$ of a continuously differentiable {rate function} $\mathcal{K}_{y\to y'}: \mathbb{R}^\mathcal{S}_{{\geq} 0} \to \mathbb{R}_{\ge 0}$ such that this positivity condition holds:
$\mathcal{K}_{y\to y'}(c) > 0$ if and only if ${\rm{supp \ }} y \subset {\rm{supp \ }} c$, where ${\rm{supp \ }} y$ refers to the support of the vector $y$, i.e., the set of species with nonzero coefficient in $y$.
The pair $\left(\mathcal{N},\mathcal{K}\right)$ is called a \emph{chemical kinetic system}.
\end{definition}

	The \emph{species formation rate function} (SFRF) of a chemical kinetic system $(\mathcal{N},\mathcal{K})$ is given by $f\left( x \right) = \displaystyle \sum\limits_{{y} \to {y'} \in \mathcal{R}} {{\mathcal{K}_{{y} \to {y'}}}\left( x \right)\left( {{y'} - {y}} \right)}$ where $x$ is a vector of concentrations of the species that change over time.
This can be written as
$f(x) = N\mathcal{K}(x)$ where $N$ is the stoichiometric matrix of $\mathcal{N}$ and $\mathcal{K}$ is the vector of rate functions.
The system of \emph{ordinary differential equations} (ODEs) of a chemical kinetic system is given by $\dfrac{{dx}}{{dt}} = f\left( x \right)$.

\begin{definition}
A \emph{positive steady state} is a positive vector $c$ of species concentrations such that $f(c)=0$.
\end{definition}

Thus, the {set of positive steady states} of a chemical kinetic system $\left(\mathcal{N},\mathcal{K}\right)$ is given by
\begin{equation}
    \label{eqn:pl-rdk-ss}
    E_+:={E_ + }\left(\mathcal{N},\mathcal{K}\right)= \left\{ {x \in \mathbb{R}^m_{>0}|f\left( x \right) = 0} \right\}.
\end{equation}

\subsection{Power law systems}

To introduce power law systems, we first define power law kinetics as follows. 

 \begin{definition}
	A \emph{power law kinetics} is a kinetics of the form $${\mathcal{K}_i}\left( x \right) = {k_i}\prod\limits_j {{x_j}^{{F_{ij}}}}  := {k_i}{{x^{{F_{i}}}}} $$ for each reaction $i =1,\ldots,r$ where ${k_i} \in {\mathbb{R}_{ > 0}}$ and ${F_{ij}} \in {\mathbb{R}}$. The $r \times m$ matrix (i.e., indexed by reactions and species) $F=\left[ F_{ij} \right]$ is called the \emph{kinetic order matrix} that contains the kinetic order values, and $k \in \mathbb{R}^r$ is called the \emph{rate vector}.
	\label{def:power:law}
\end{definition}

In particular, if the kinetic order matrix contains the stoichiometric coefficients of reactant $y$ for each reaction $y \to y'$ in the network, then the system follows the well-known \emph{mass action kinetics}.

A \emph{power law system} is a chemical kinetic system with power law as its kinetics. Furthermore, we define a \emph{mass action system} in a similar manner.

{
\begin{remark}
Power law systems utilize the generalized mass action (GMA) format for kinetics, wherein the kinetic values are not necessarily restricted to positive integers as seen in the stoichiometric coefficients in reactant complexes of the underlying network, as in the case for mass action kinetics. The notion of GMA kinetics, introduced by M\"uller and Regensburger, restricts the kinetic values into nonnegative numbers \cite{Muller2012,Muller2014}. In our work, we employ power law kinetics as our findings are valid not only for nonnegative kinetic values but for real numbers in general.
\end{remark}
}

\begin{definition}
	A power law system has \emph{reactant-determined kinetics} (or of type \emph{PL-RDK}) if for any two reactions $i, j$ with identical reactant complexes, the corresponding rows of kinetic orders in $F$ are identical, i.e., ${F_{ik}} = {F_{jk}}$ for $k = 1,2,...,m$.
\end{definition}

{{\begin{definition}
	The $m \times n$ matrix $\widetilde{Y}$ is defined as
	$${\left( {\widetilde{Y}} \right)_{jk}} = \left\{ \begin{array}{l}
		{F_{ij}}{\text{  \ \ \ \ \ \ \  if $k$ {\text{is the reactant complex in reaction }}$i$}}\\
		{0} {\text{  \ \ \ \ \ \ \  \ \   otherwise \ }}
	\end{array} \right..$$
\end{definition}}}

\begin{definition}
	Let $n_r$ be the number of reactant complexes of a CRN. The $m \times n_r$ \emph{T-matrix} is the truncated $\widetilde{Y}$ where the non-reactant columns are deleted.
\end{definition}

In other words, each column of the {$T$-matrix} contains the kinetic order values associated with the species.


\subsection{Decomposition of chemical reaction networks}

{
A CRN can be decomposed into pieces of networks called \emph{subnetworks}, 
which we define as follows:

\begin{definition}
    Let $\mathcal{N}$ be a reaction network and $\mathcal{R}$ its reaction set.
	A {\emph{decomposition of $\mathcal{N}$}} into $\mathcal{N}_1$, $\mathcal{N}_2$, \ldots, $\mathcal{N}_\alpha$ is induced by partitioning $\mathcal{R}$ into $\mathcal{R}_1,  \mathcal{R}_2, \ldots, \mathcal{R}_\alpha$, respectively. The resulting networks $\mathcal{N}_1$,  $\mathcal{N}_2$, \ldots, $\mathcal{N}_\alpha$ are called \emph{subnetworks} of $\mathcal{N}$.
\end{definition}

We follow the notion of subnetworks introduced by M. Feinberg \cite{feinberg2019crnt}. That is, each $\mathcal{N}_i$ has the same set of species as $\mathcal{N}$, although some of the species seem to have no role in this particular subnetwork. Furthermore we consider $\mathcal{N}_i = (\mathcal{S},\mathcal{C}_i,\mathcal{R}_i)$ where $\mathcal{C}_i$ is the set of all complexes that appear as a product or reactant of a reaction in the $i$th subnetwork \cite{feinberg2019crnt}.
}


\begin{definition}
If the rank of the stoichiometric matrix of the whole network is the sum of the ranks of the stoichiometric matrices of its subnetworks, then the decomposition is \emph{independent}. These subnetworks are called \emph{independent subnetworks}.
\end{definition} 

This notion of independent decomposition lends itself to great importance in our study, since it gives an important relationship between the structure of the set of positive steady states of a given network and its independent subnetworks. The following theorem is due to Martin Feinberg \cite[Appendix 6.A]{feinberg2019crnt}.
\begin{theorem}[Feinberg Decomposition Theorem]
    \label{thm:feinberg-decomp}
    Let $(\mathcal N, \mathcal K)$ be a chemical kinetic system composed of the reaction network $\mathcal N$ and corresponding kinetics $\mathcal K.$ Suppose $\mathcal N$ is decomposed into $\alpha$ subnetworks, say $\mathcal N_1, \mathcal N_2, \dots, \mathcal N_k$ and denote the restriction of $\mathcal K$ to the restrictions in $\mathcal N_i$ as $\mathcal K_i.$ Then \[\bigcap_{i=1}^\alpha E_+(\mathcal N_i, \mathcal K_i) \subseteq E_+(\mathcal N, \mathcal K)\] where $E_+(\mathcal N, \mathcal K)$ denotes the set of positive steady states of the whole system and $E_+(\mathcal N_i, \mathcal K_i)$ denotes the set of positive steady states of the $i$th subsystem. Moreover, if the network decomposition is independent, then we have \[\bigcap_{i=1}^\alpha E_+(\mathcal N_i, \mathcal K_i) = E_+(\mathcal N, \mathcal K).\]
\end{theorem}

Returning to our discussion on power law systems, suppose that $\mathcal{N}_i$ denotes the $i$th subnetwork for $i=1,2,\ldots,\alpha$. Furthermore, let $e_1,e_2,\ldots,e_\alpha \in \{0,1\}^n$ be the characteristic vectors of subsets of $\mathcal{C}_1,\mathcal{C}_2,\ldots,\mathcal{C}_\alpha$, respectively, where $\mathcal{C}_i$ is the set of reactant complexes in subnetwork $\mathcal{N}_i$. From the idea of B. Boros, we construct an $n_r \times \alpha$ matrix $L=[e_1,e_2,\ldots,e_\alpha]$.

We define the $\widehat{T}$-matrix for a network $\mathcal{N}$ as
$\begin{bmatrix}
    T\\
    L^{\top}
\end{bmatrix}$.
Moreover, for
$\widehat T_i$ associated with subnetwork $\mathcal{N}_i$, we have $\begin{bmatrix}
    T_i\\
    e_i^{\top}
\end{bmatrix}$.
Hence, the $\widehat{T}$-matrix is the augmented matrix containing the kinetic order vectors in the network.


 \section{Non-emptiness of the set of positive steady states of power law systems with independent decompositions}
	\label{result1}
{
Before we begin with the main result, we first discuss additional concepts and an important result to proceed with the proof of the theorem.

	\begin{definition}
		The \emph{molecularity matrix} or \emph{matrix of complexes} $Y$ is an $m\times n$ matrix such that $Y_{ij}$ is the stoichiometric coefficient of species $i$ in a complex $j$.
	\end{definition}

\begin{definition}
	The \emph{Laplacian matrix} $L$ associated to the graph of a kinetic system is an $n\times n$ matrix such that \[L_{ij} = \begin{cases}
		k_{ji}&\text{if }i \neq j\\
		k_{jj} - \displaystyle \sum_{l = 1}^n k_{jl}&\text{if }i = j
	\end{cases}\] where $k_{ji}$ is the rate constant for each reaction $i\to j.$
\end{definition}

 Note that if the reaction $i\to j$ is not present in the network, then $k_{ji} = 0.$ The Laplacian matrix is a common descriptor for graphs, as it quantifies how each vertex (complex) in a network differs from other vertices through its edges (reactions).

\begin{definition}
	The \emph{factor map} $\theta\colon \mathbb R^m\to \mathbb R^n$ of a kinetic system is defined as  
	\[\theta_C(x) = \begin{cases}
		\displaystyle \prod\limits_{j=1}^m x_i^{F_{ij}} = x^{F_i}&\text{if }C\text{ is a reactant complex of reaction }i\\
		0&\text{otherwise}
	\end{cases}.\]
\end{definition}

We use these concepts to redefine the set of positive steady states in \eqref{eqn:pl-rdk-ss}. As such, we have
\begin{equation}
    \label{def:steady_state1}
    E_+ = \{x\in\mathbb R_{>0}^m ~|~ Y\cdot L\cdot \theta(x) = 0\}
\end{equation} where $Y, L,$ and $\theta$ are the matrix of complexes, Laplacian matrix, and factor map of the kinetic system, respectively.

This was taken from the work of B. Boros \cite{boros:paper:thesis,boros:thesis}, which is a corollary of the Farkas' Lemma.

\begin{corollary}
    \label{thm:lin-alg-corollary}
    Let $\alpha, m,$ and $c_1, c_2, \ldots, c_\alpha$ be positive integers. Let $A_i\in\mathbb R^{c_i\times m}$ and $b_i\in\mathbb R^{c_{i}}$ for $i = 1,\ldots, \alpha.$ Suppose that
    \begin{enumerate}
        \item $\{x\in\mathbb R^m\mid A_ix=b_i\}\neq\emptyset$ for $i = 1,\ldots, \alpha,$ and 
        \item $\mathrm{Im}\,[A_1^\top, A_2^\top, \ldots, A_\ell^\top] = \mathrm{Im}\,A_1^\top \oplus \mathrm{Im}\,A_2^\top \oplus \cdots \oplus \mathrm{Im}\,A_\alpha^\top.$
    \end{enumerate} Then $\displaystyle \bigcap\limits_{i=1}^\alpha \{x\in\mathbb R^m\mid A_ix = b_i\}\neq\emptyset.$
\end{corollary}
}


  We are now ready to present {the main result} of this paper on the non-emptiness of the set of positive steady states of a class of power law systems. Furthermore, we will present interesting examples to illustrate this result.

 \begin{theorem}
        \label{thm:main-result}
        Let $(\NN, \KK)$ be a {power law} system decomposed into $\alpha\in\mathbb N$ {subsystems.} Suppose the {system's underlying network} satisfies
        \begin{enumerate}
            \item {$S = S_1 \oplus S_2 = \oplus \ldots \oplus S_\alpha,$ i.e. stoichiometric independence}
            \item $\widehat T = \widehat T_1 \oplus \widehat T_2 \oplus \ldots \oplus \widehat T_\alpha,$ {i.e. $\widehat T$-independence.}
        \end{enumerate} Then $E_+(\NN,\KK)\neq\emptyset$ if and only if $E_+(\NN_i,\KK_i)\neq\emptyset$ for each $i = 1, 2, \dots, \alpha.$

        \begin{proof}
            Let $\alpha\in\mathbb N.$ Suppose $E_+(\NN,\KK)\neq\emptyset.$ By Theorem \ref{thm:feinberg-decomp}, we have \[E_+(\NN,\KK) \subseteq \bigcap_{i=1}^{\alpha}E_+(\NN_i,\KK_i).\] Since the decomposition of $(\NN,\KK)$ is independent, we achieve set equality. It follows directly that $E_+(\NN_i,\KK_i)\neq\emptyset$ for each $i = 1,2,\ldots,\alpha.$

            For the converse, we assume that $E_+(\NN_i,\KK_i)\neq\emptyset$ for all $i = 1,2,\ldots,\alpha.$ Now, from \eqref{def:steady_state1}, in 
            we have $x\in E_+(\NN_i,\KK_i)$ if and only if \[\text{there exists } v_i \in \mathbb{R}_{>0}^{n_i} \text{ such that } \theta_i(x) = v_i \text{ and } Y_i \cdot L_i \cdot v_i = 0\] where $v_i$ is a vector with positive real values and whose size is equal to the number of complexes ($n_i$) in the subnetwork $\NN_i, Y_i$ is the molecularity matrix, and $L_i$ is the Laplacian matrix for the subnetwork $\NN_i$. 

            We define the element-wise $\ln$ function for vectors of size $n$ as
            \begin{equation*}
                \begin{aligned}
                    \ln& \colon \qquad \mathbb{R}_{>0}^m &&\to \qquad \mathbb{R}^m\\
                    &\phantom{\colon} (y_1, y_2, \ldots, y_m)^{\top} &&\mapsto (\ln y_1, \ln y_2, \ldots, \ln y_m)^{\top}.
                \end{aligned}
            \end{equation*} Note that this function is bijective. Then \[\ln(\theta_i(x)) = (T_i)^{\top}\ln(x)\] where $T_i$ is the matrix of reactant complexes of the subnetwork $\NN_i.$ Hence, $\ln(\theta_i(x)) = \ln(v_i)\in\mathrm{Im}(T_i)^{\top}.$ Thus, we can also say that $x\in E_+(\NN_i,\KK_i)$ if and only if \[\text{there exists } v_i\in\mathbb{R}_{>0}^{n_i}\text{ such that }\ln(v_i)\in\mathrm{Im}(T_i)^{\top}\text{ and } Y_i \cdot L_i \cdot v_i = 0.\]
            
            Define $\mathbf{1}_i\in \mathbb{R}^{n_i}$ as a vector with all coordinates equal to one.  Note that $\mathrm{Im}\left(\widehat T_i\right)^\top = \mathrm{Im}[T_i, \mathbf{1}_i]$ and for all $\gamma_i\in\mathbb{R}_{>0},$ we have $\ln(\gamma_iv_i) = \ln(v_i) + \ln(\gamma_i)\mathbf 1_i.$ Thus $\ln(v_i)\in\mathrm{Im}(T_i)^\top$ if and only if $\ln(v_i)\in\mathrm{Im}(\widehat T_i)^\top.$ This gives us the equivalence $x\in E_+(\NN_i,\KK_i)$ if and only if
            \[\text{there exists } v_i\in\mathbb{R}_{>0}^{n_i}\text{ such that }\ln(v_i)\in\mathrm{Im}(\widehat T_i)^{\top}\text{ and } Y_i \cdot L_i \cdot v_i = 0.\]
            
            Having assumed that $E_+(\NN_i,\KK_i)\neq\emptyset$ for all $i = 1, 2, \ldots, \alpha,$ we will now fix $v_1, v_2, \dots, v_\alpha$ such that $\ln(v_i)\in\mathrm{Im}(\widehat T_i)^{\top}$ and $Y_i\cdot L_i\cdot v_i = 0$ for all $i = 1, 2, \ldots, \alpha.$ Hence \[\left\{z\in\mathbb{R}^{m+1}_{>0}~|~\left(\widehat T_i\right)^{\top}\cdot z = \ln(v_i)\right\}\neq\emptyset\] for every $i = 1,2,\ldots, \alpha.$ Since, by assumption, we have $\widehat T = \widehat T_1 \oplus \widehat T_2 \oplus \ldots \oplus \widehat T_\alpha,$ then by Corollary \ref{thm:lin-alg-corollary},
            we get \[\bigcap_{i=1}^\alpha \left\{ z\in \mathbb{R}_{>0}^{m+1}~|~ \left(\widehat T_i\right)^\top\cdot z = \ln(v_i)\right\} \neq \emptyset.\]


            Hence, there exist $u\in\mathbb{R}^n$ and $w\in\mathbb{R}^\alpha$ such that \[\bgroup \renewcommand*{\arraystretch}{1.5}\begin{bmatrix}
                ~\left(\widehat T_1\right)^\top~\\
                \left(\widehat T_2\right)^\top\\
                \vdots\\
                \left(\widehat T_\alpha\right)^\top
            \end{bmatrix}\egroup\begin{bmatrix}u\\w\end{bmatrix} = \begin{bmatrix}
                \ln(v_1)\\
                \ln(v_2)\\
                \vdots\\
                \ln(v_\alpha)
            \end{bmatrix}\] Take $x\in\mathbb{R}^n_{>0}$ and  $\gamma_j\in\mathbb{R}_{>0}$ such that $\ln(x) = u$ and $\ln(\gamma_j) = -w_j$ for all $j = 1, 2,\ldots, \alpha.$ Then \[\begin{bmatrix}u\\w\end{bmatrix} = \begin{bmatrix}\ln (x)^\top\\\hline
            -\ln(\gamma_1)\\
            -\ln(\gamma_2)\\
            \vdots\\
            -\ln(\gamma_\alpha)\end{bmatrix} \implies \bgroup \renewcommand*{\arraystretch}{1.5}\begin{bmatrix}
                ~\left(\widehat T_1\right)^\top~\\
                \left(\widehat T_2\right)^\top\\
                \vdots\\
                \left(\widehat T_\alpha\right)^\top
            \end{bmatrix}\egroup\begin{bmatrix}\ln (x)^\top\\
            -\ln(\gamma_1)\\
            -\ln(\gamma_2)\\
            \vdots\\
            -\ln(\gamma_\alpha)\end{bmatrix} = \begin{bmatrix}
                \ln(v_1)\\
                \ln(v_2)\\
                \vdots\\
                \ln(v_\alpha)
            \end{bmatrix}.\] Thus, for all $i = 1, 2, \ldots, \alpha$ and $y\in\CC_i,$ we get \[\gamma_i(v_i)_y = \prod_{k=1}^\alpha x_k^{F_{ky}} = (\theta_i)_y(x)\] which gives $\gamma_iv_i = \theta_i(x).$ As a result, we have for $x\in\mathbb R_{>0}^m$ that
            \begin{align*}
                Y\cdot L \cdot \theta(x)
                &= \sum_{i=1}^{\alpha}Y_i\cdot L_i\cdot\theta_i(x)\\
                &= \sum_{i=1}^{\alpha} Y_i\cdot L_i\cdot(\gamma_i\cdot v_i)\\
                &= \sum_{i=1}^\alpha \gamma_i\cdot\left(Y_i\cdot L_i\cdot v_i\right)\\
                &= \sum_{i=1}^\alpha \gamma_i\cdot 0 = 0.
            \end{align*} Hence, by \eqref{def:steady_state1}, we get $x\in E_+(\NN, \KK).$ Therefore $E_+(\NN, \KK)\neq\emptyset$ which completes our proof.
        \end{proof}
\end{theorem}

\begin{example} 
    {We c}onsider the PL-RDK system for the Earth's pre-industrial carbon cycle $(\NN, \KK)$ as presented in \cite{anderies2013topology,fortun2018deficiency}. {The system's network has three species, namely $X_1, X_2,$ and $X_3,$ denoting the land, atmosphere, and ocean carbon pools, respectively. It has six complexes given by $X_1 + 2X_2, 2X_1 + X_2, X_1 + X_2, 2X_2, X_2,$ and $X_3,$ and its reaction set is given by $\RR = \{R_1, R_2, R_3, R_4\}$ which consists of four reactions: $R_1 = (X_1 + 2X_2, 2X_1 + X_2), R_2 = (X_1 + X_2, 2X_2), R_3 = (X_2, X_3),$ and $R_4 = (X_3, X_2)$. The graph of $\NN$ is given by
    
    
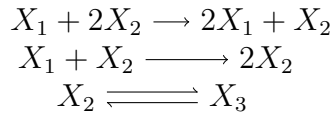
\begin{equation*}
		\begin{tikzpicture}[baseline=(current  bounding  box.center)]
			\node (X1 + 2X2) {$X_1 + 2X_2$};
			\node (2X1 + X2) [right of=X1 + 2X2, node distance = {25mm}]{$2X_1 + X_2$};
			\node (X1 + X2) [below of=X1 + 2X2, , node distance = {5mm}]{$X_1 + X_2$};
			\node (2X2) [right of=X1 + X2, node distance = {25mm}]{$2X_2$};
			\node (X2) [below of=X1 + X2, , node distance = {5mm}]{$X_2$};
			\node (X3) [right of=X2, node distance = {20mm}]{$X_3$};
			\draw [->] (X1 + 2X2) -- (2X1 + X2);
			\draw [->] (X1 + X2) -- (2X2);
			\draw [-left to] ($(X2.east) + (0pt, 2pt)$) -- ($(X3.west) + (0pt, 2pt)$);
			\draw [-left to] ($(X3.west) + (0pt, -2pt)$) -- ($(X2.east) + (0pt, -2pt)$);
		\end{tikzpicture}
	\label{graph:example_preind}
	\end{equation*}}
    {Moreover,} the network's $T$-matrix is given by \begin{equation*}
        T = \begin{blockarray}{ccccc}
		X_1 + 2X_2 & 2X_1 + X_2 & X_2 & X_3 \\
		\begin{block}{[cccc]c}
			f_{11}&f_{12}&0&0&X_1\\
            f_{21}&f_{22}&1&0&X_2\\
            0&0&0&1&X_3\\
        \end{block}
    \end{blockarray}\end{equation*} where $f_{11}, f_{12}, f_{21}, f_{22}\in\mathbb R.$ Now, suppose either $f_{11} \neq f_{12}$ or $f_{21}\neq f_{22}.$ Using the methodology in \cite[p.~13]{hernandez2021independent}, we can generate the independent decomposition of $\NN$ to $\NN_1$ and $\NN_2$ whose reaction sets are given by {$\RR_1 = \{R_1, R_2\}$ and $\RR_2 = \{R_3, R_4\},$} respectively. {Since the decomposition is independent, we have by definition that the direct sum of the stoichiometric matrices of the subnetworks is precisely the stoichiometric matrix of the entire network.}
    
    


    {We will now show that the system satisfies $\widehat T$-independence, i.e. $\widehat T = \widehat T_1 \oplus \widehat T_2.$ To prove this, we show that the sum of the ranks of the $\widehat T$ matrices of the subnetworks equals the rank of the $\widehat T$ matrix of the entire network. } 

    Recall that $\widehat T$ is the augmented matrix of kinetic order vectors and for $\NN,$ this is given by
    \[\widehat T = \begin{blockarray}{ccccc}
		X_1 + 2X_2 & 2X_1 + X_2 & X_2 & X_3 \\
		\begin{block}{[cccc]c}
			f_{11}&f_{12}&0&0&X_1\\
            f_{21}&f_{22}&1&0&X_2\\
            0&0&0&1&X_3\\
            1&1&0&0&\NN_1\\
            0&0&1&1&\NN_2\\
        \end{block}
    \end{blockarray}.\] Notice that if $f_{11}\neq f_{12},$ then the $X_1$ row must be linearly independent with the $\NN_1$ row. We also notice that the $X_2, X_3,$ and $\NN_2$ rows are linearly independent due to the presence of zero elements. Therefore, {we have ${\rm{rank}}(\widehat T) = 4.$}

    We proceed to the augmented matrix of kinetic order vectors for subnetworks $\NN_1$ and $\NN_2.$ For $\NN_1,$ we have \[\widehat T_1 = \begin{blockarray}{ccc}
		X_1 + 2X_2 & 2X_1 + X_2\\
		\begin{block}{[cc]c}
			f_{11}&f_{12}&X_1\\
            f_{21}&f_{22}&X_2\\
            1&1&\NN_1\\
        \end{block}
    \end{blockarray}\] with rank two since either the $X_1$ or $X_2$ row must be linearly independent with the $\NN_1$ row. For $\NN_2,$ we have \[\widehat T_2 = \begin{blockarray}{ccc}
		X_2 & X_3\\
		\begin{block}{[cc]c}
			1&0&X_2\\
            0&1&X_3\\
            1&1&\NN_1\\
        \end{block}
    \end{blockarray}\] which also has rank two. Since $\mathrm{rank}( \widehat T_1) + \mathrm{rank}(\widehat T_2) = 2 + 2 = 4 = \mathrm{rank}(\widehat T),$ we achieve $\widehat T$-independence. Therefore, the system satisfies the conditions of Theorem \ref{thm:main-result}. Thus, the non-emptiness of the set of positive steady states of the subnetworks $\mathcal{N}_1$ and $\mathcal{N}_2$ implies the non-emptiness of the set of positive steady states of the whole network $\mathcal{N}$.

    Indeed, the case for when $f_{21}\neq f_{22}$ and $f_{11} = f_{12}$ has been explored by Hernandez and Buendicho in \cite{hernandez2023network}. Under these assumptions, they were able to analytically solve for the parameterized positive steady states of the entire network by merging the positive steady states of the given subnetworks.
\end{example}


{
\begin{remark}
    If a power law system satisfies the conditions of independent decomposition to $\alpha$ subnetworks and $\widehat T$-independence, then Theorem \ref{thm:main-result} tells us that the existence of a single subnetwork with no positive steady states implies that the entire system will also have no positive steady states.
\end{remark}
}

\begin{example}
    Consider the toy system $(\NN,\KK)$ whose corresponding network has the reaction set given by $\RR = \{R_1, R_2, R_3, R_4\}$ where \begin{align*}
        R_1&\colon 2X\to X, && R_3\colon X + Y \to X,\\
        R_2&\colon 2Y\to X + 2Y,&&
        R_4\colon Y\to 2Y.
    \end{align*} The network has the corresponding graph
    
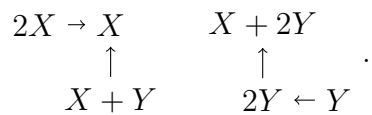
\begin{equation*}
		\begin{tikzpicture}[baseline=(current  bounding  box.center)]
			\node (X+Y) {$X + Y$};
            \node (X) [above of=X+Y]{$X$};
            \node (2X) [left of=X]{$2X$};
			\draw [->] (X+Y) -- (X);
			\draw [->] (2X) -- (X);
            \node (2Y) [right of= X+Y, node distance={20mm}]{$2Y$};
            \node (Y) [right of=2Y] {$Y$};
            \node (X+2Y) [above of=2Y]{$X + 2Y$};
            \draw[->] (Y) -- (2Y);
            \draw[->] (2Y) -- (X+2Y);
        \end{tikzpicture}.
	\end{equation*}
 Moreover, suppose the system is endowed with PL-RDK whose kinetic-order matrix is given by
    \[F = \begin{blockarray}{ccc}
		X & Y\\
		\begin{block}{[cc]c}
			2&0&R_1\\
            0&2&R_2\\
            1/3&2/3&R_3\\
            0&1&R_4\\
        \end{block}
    \end{blockarray}.\] 

    Using the methods in \cite[p.~13]{hernandez2021independent}, we can generate the independent decomposition $\NN = \NN_1\cup \NN_2$ where the reaction sets of $\NN_1$ and $\NN_2$ are given by $\RR_1 = \{R_1, R_2\}$ and $\RR_2 = \{R_3, R_4\},$ respectively. 
    
    For each reaction, the kinetic vectors and rate laws are given as follows.
    \[\begin{array}{lcc}
        \hline
        \text{Reaction}&\text{Kinetic Vector}&\text{Rate Law}\\\hline
        R_1\colon 2X \to X & [2,0]^\top = 2X& k_1x^2y^0 = k_1x^2\\
        R_2\colon 2Y \to X + 2Y& [0,2]^\top = 2Y & k_2x^0y^2 = k_2y^2\\
        R_3\colon X + Y\to X& [1/3, 2/3]^\top = \frac{1}{3}X + \frac{2}{3}Y& k_3x^{1/3}y^{2/3}\\
        R_4\colon Y\to 2Y & [0, 1]^\top = Y & k_4x^0y^1 = k_4y\\\hline
    \end{array}\] where $k := [k_1, k_2, k_3, k_4]^\top \in \mathbb R_{>0}^4$ is the rate vector of the system. Then, the species formation rate function (SFRF) of the system is given by \[f(x,y) = \begin{bmatrix}
        \dot x\\ \dot y
    \end{bmatrix} = k_1x^2\begin{bmatrix}
        -1\\0
    \end{bmatrix} + k_2y^2 \begin{bmatrix}
        1\\0
    \end{bmatrix} + k_3x^{1/3}y^{2/3}\begin{bmatrix}
        0\\-1
    \end{bmatrix} + k_4y\begin{bmatrix}
        0\\1
    \end{bmatrix}\] where $x$ and $y$ are the concentrations of $X$ and $Y$ in the system, respectively.

    Now, we investigate the SFRFs of each of the subnetworks. For $\NN_1,$ the SFRF is given by \[f_1(x,y) = \begin{bmatrix}-k_1x^2 + k_2y^2\\0\end{bmatrix}\] which equals $[0,0]^\top$ whenever \[\begin{cases}
        x = \sqrt{\dfrac{k_1}{k_2}}\tau_1\\
        y = \tau_1
    \end{cases},\quad \tau_1\in\mathbb R_{>0},\] which is a positive steady state parametrization for $\NN_1$. Here, we have $y$ to be a free parameter. For $\NN_2,$ we have \[f_2(x,y) = \begin{bmatrix} 0\\-k_3x^{1/3}y^{2/3} + k_4y\end{bmatrix}\] Here, we have $f_2(x,y) = 0$ whenever
    \[-k_3x^{1/3}y^{2/3} + k_4y = y^{2/3}(-k_3x^{1/3} + k_4y^{1/3}) = 0.\] Since we are concerned with the system's positive steady states, we have $f_2(x,y) = 0$ when $-k_3x^{1/3} + k_4y^{1/3} = 0.$ Then, taking $\tau_2\in\mathbb R_{>0}$ to be a free parameter for the concentration $y$ at steady state, we have $(x,y)$ to be the positive steady states of $\NN_2$ whenever
    \[\begin{cases}
        x = \left(\dfrac{k_4}{k_3}\right)^{1/3}\tau_2\\
        y = \tau_2
    \end{cases},\quad \tau_2\in\mathbb R_{>0},\] which is a positive steady state parametrization for $\NN_2$. 

    Note that the set of positive steady states of $\NN_1$ and $\NN_2$ are nonempty for any rate vector $k\in \mathbb R^4_{>0}.$ However, if we take the intersection of these positive steady states to obtain that of the entire network, then we get \[\left(\sqrt{\dfrac{k_1}{k_2}}\tau_1, \tau_1\right) = \left(\left(\dfrac{k_4}{k_3}\right)^{1/3}\tau_2, \tau_2\right).\] This gives $\tau_1 = \tau_2$ and so \[\sqrt{\dfrac{k_1}{k_2}} = \left(\dfrac{k_4}{k_3}\right)^{1/3} \implies \frac{k_1}{k_2} = \left(\dfrac{k_4}{k_3}\right)^{2/3}.\] Therefore, there exists some rate vectors for which the set positive steady states of the whole network becomes empty.


    {According to Theorem \ref{thm:main-result}, since the decomposition is independent, the network does not satisfy the condition of $\widehat T$-independence. We show that this is indeed the case.}

    The reactant complexes of our system is given by $\{2X, 2Y, X + Y, Y\}\subseteq \CC.$ From this, we have the augmented matrix of kinetic order vectors of $\NN$ given by \[\widehat T = \begin{blockarray}{ccccc}
		2X & 2Y & X + Y & Y\\
		\begin{block}{[cccc]c}
			2&0&1/3&0&X\\
            0&2&2/3&1&Y\\
            1&1&0&0&\NN_1\\
            0&0&1&1&\NN_2\\
        \end{block}
    \end{blockarray}\] which is of rank three. For the independent subsystems $\NN_1$ and $\NN_2$, we have their augmented matrix of kinetic order vectors given by \[\widehat T_1 = \begin{blockarray}{ccc}
		2X & 2Y\\
		\begin{block}{[cc]c}
			2&0&X\\
            0&2&Y\\
            1&1&\NN_1\\
            0&0&\NN_2\\
        \end{block}
    \end{blockarray},\quad \widehat T_2 = \begin{blockarray}{ccc}
		X + Y & Y\\
		\begin{block}{[cc]c}
			1/3&0&X\\
            2/3&1&Y\\
            0&0&\NN_1\\
            1&1&\NN_2\\
        \end{block}
    \end{blockarray},\] respectively. The ranks of $\widehat T_1$ and $\widehat T_2$ are equal to two. Since $2+2\neq 3$, the ranks of the $\widehat T$ matrices of the independent subnetworks do not add up to that of the entire network. Therefore, we do not achieve $\widehat T$-independence, as desired.
    
\end{example}

 \section{Derivation of steady state parametrizations of power law systems with identical independent subnetworks}
	\label{result2}
 
{
 Before we discuss the next result, we first define what we mean when we say that two networks are \emph{identical.}

\begin{definition}
    Two networks $\NN$ and $\NN'$ are said to be \emph{identical} if there exist one-to-one correspondences between the sets of complexes and reactions in $\NN$ and the sets of complexes and reaction in $\NN',$ respectively.
\end{definition}

As an example, consider the networks $\NN$ and $\NN'$ whose graphs are given by $X\rightleftharpoons Y$ and $Y\rightleftharpoons Z,$ respectively. The two networks are identical since we can form the mapping between the sets of complexes of $\NN$ and $\NN'$ which maps the complex $X$ to $Y$ and the complex $Y$ to $Z.$ We also have the mapping which maps the reaction $X\rightarrow Y$ in $\NN$ to $Y\rightarrow Z$ in $\NN',$ as well as that of $Y\rightarrow X$ in $\NN$ to $Z\rightarrow Y$ in $\NN'.$
}

The recent study of Villareal et al. \cite[Section~3.5]{Villareal2023} presented a methodology to solve parametrized positive steady states of a network, under mass action kinetics, that can be decomposed to independent and identical subnetworks. Here, we generalize the method to allow for PL-RDK systems that satisfy the assumptions of Theorem \ref{thm:main-result}, which guarantees the existence of positive steady states of the {system} whenever each independent {subsystem} has positive steady states.

{Recall that PL-RDK systems are systems with power law kinetics such that the rows of kinetic orders for the same reactant complexes are equal.} {Using \autoref{thm:main-result},} we can now introduce the steps to derive a positive steady state parametrization of a CRN with PL-RDK kinetics and having $n$ independent and identical subnetworks, where $n$ is any positive integer.
\begin{itemize}
    \item[1.] Decompose the underlying CRN into $n$ independent and identical subnetworks.
    \item[2.] Derive the positive steady state parametrization for any independent subnetwork, with kinetics dictated by the kinetic order vectors. If the subnetwork has a positive steady state, then proceed with the next step. Otherwise, there is no need to proceed to the next step because it will follow that the whole network has no positive steady state. 
    \item[3.] Generalize a formula for parametrized positive steady states for any subnetwork based on the results in the previous step.
    \item[4.] Derive a positive steady state parametrization of the network for the case when $n=2$ by merging the positive steady states of two independent subnetworks.
    \item[5.] Apply the principle of mathematical induction to obtain a positive steady state parametrization of the whole network for any positive integer $n$.
\end{itemize}
 It is easy to show that the steps are valid by invoking the Feinberg Decomposition Theorem (i.e., Theorem \ref{thm:feinberg-decomp}) and then using the Principle of Mathematical Induction.

 
 To illustrate the steps, we consider a network $\NN^n$ with $n+1$ species, $n+1$ complexes, and $n$ reversible reactions whose graph is given by  
 \begin{equation}
 \label{eqn:futile-cyc}
 \begin{tikzpicture}[baseline=(current  bounding  box.center)]
			\node (X0) {$X_0$};
            \node (X1) [right of=X0] {$X_1$};
            \node (dots)[right of=X1]{$\ldots$};
            \node (Sn) [right of=dots]{$X_n$};

            \draw [-left to] ($(X0.east) + (0pt, 1pt)$) -- ($(X1.west) + (0pt, 1pt)$);
            \draw [-left to] ($(X1.west) + (0pt, -1pt)$) -- ($(X0.east) + (0pt, -1pt)$);

            \draw [-left to] ($(X1.east) + (0pt, 1pt)$) -- ($(dots.west) + (0pt, 1pt)$);
            \draw [-left to] ($(dots.west) + (0pt, -1pt)$) -- ($(X1.east) + (0pt, -1pt)$);

            \draw [-left to] ($(dots.east) + (0pt, 1pt)$) -- ($(Sn.west) + (0pt, 1pt)$);
            \draw [-left to] ($(Sn.west) + (0pt, -1pt)$) -- ($(dots.east) + (0pt, -1pt)$);
\end{tikzpicture}.\end{equation}
Here, we will solve a positive steady state parametrization symbolically.

Notice that the $n$ reversible reactions give us a total of $2n$ reactions. In this scenario, we can decompose $\NN^n$ to $n$ independent and identical subnetworks, each denoted $\NN_i^n$ whose reaction set is given by $\RR_i = \{X_i \to X_{i+1}, X_{i+1} \to X_i\}$ for $i = 0,1,\ldots, n-1.$ If the kinetic order associated to each species equals one, then we have the familiar mass action system. 

Suppose that the kinetic order associated to the reactant species $X_i$ in reaction $X_i\to X_{i+1}$ equals $f_i$ for $i = 0,1,\ldots, n-1.$ Meanwhile, suppose that the kinetic order of $X_{j}$ in the backward reaction $X_{j}\to X_{j-1}$ is $f_j'$ for $j = 1,2\ldots, n.$ Since the system is reactant-determined, we must have $f_i = f_j'$ for all $i = j.$ Moreover, suppose the rate constant for the reactions $X_{i}\to X_{i+1}$ is $k_{i}$ for $i = 0,1,\ldots, n-1.$ For the reverse reaction $X_{i+1}\to X_i,$ we assume the rate constant to be $k_{i}'$ for $i = 0,1,\ldots, n-1.$


{Now, we will show that the system satisfies the conditions in \autoref{thm:main-result}. Since we have $\NN^n$ to $n$ independent and identical subnetworks, the system's underlying network satisfies the first condition. It remains to show that such decomposition will also satisfy $\widehat T$-independence.}


We can obtain the $T$-matrix, which is the truncated matrix of kinetic order values with non-reactant columns deleted, of the system given by \[T = \begin{bmatrix}
    f_0 & 0 & 0 & 0 & 0 & \cdots & 0 & 0 & 0\\
    0 & f_1 & f_1' & 0 & 0 & \cdots & 0 & 0 & 0\\
    0 & 0 & 0 & f_2 & f_2' &\cdots & 0 & 0 & 0\\
    \vdots & \vdots & \vdots &\vdots & \vdots & \ddots & \vdots & \vdots & \vdots\\
    0 & 0 & 0 & 0 & 0 & \cdots & f_{n-1} & f_{n-1}' & 0\\
    0 & 0 & 0 & 0 & 0 & \cdots & 0 & 0 & f_n
\end{bmatrix}.\] Note that the size of this matrix is $2n\times(n+1).$ If we augment this $T$-matrix with the characteristic vector of each reactant complex in the $n$ subnetworks, we generate the $\widehat T$-matrix of size $2n\times (2n+1)$ for the network. Since each of the columns are linearly independent, we conclude that the rank of the $\widehat T$-matrix is $2n.$ Note that $f_i = f_{i}'$ for each $i = 1,\ldots, n-1$ since we are dealing with PL-RDK systems. Since any subnetwork in the system involves only two reactions, the rank of each $\widehat T$-matrix of the $n$ subnetworks is two. {Note here that reactant determined kinetics ensures the independence of the $\widehat T$ matrices of the subnetworks.} This shows that the system satisfies the second assumption. Since the network satisfies the two conditions presented in \autoref{thm:main-result}, we can take the intersection of the positive steady states of each subnetwork to obtain the positive steady states for the entire network. 


{In addition to the network we consider in \eqref{eqn:futile-cyc}, numerous futile cycles documented in various literature sources \cite{conradi2018dynamics,feliu2020proof,rao2017global,reyes2022numerical} could be explored to demonstrate the applicability of our approach in analyzing such networks.}

\begin{theorem}
    \label{thm:result-2}
    For every $n\in\mathbb N,$ the PL-RDK system $(\NN^n, \KK^n)$ (or simply $\NN^n$) with underlying CRN 
    \begin{center}\begin{tikzpicture}[baseline=(current  bounding  box.center)]
			\node (X0) {$X_0$};
            \node (X1) [right of=X0] {$X_1$};
            \node (dots)[right of=X1]{$\ldots$};
            \node (Xn) [right of=dots]{$X_n$};

            \draw [-left to] ($(X0.east) + (0pt, 1pt)$) -- ($(X1.west) + (0pt, 1pt)$);
            \draw [-left to] ($(X1.west) + (0pt, -1pt)$) -- ($(X0.east) + (0pt, -1pt)$);

            \draw [-left to] ($(X1.east) + (0pt, 1pt)$) -- ($(dots.west) + (0pt, 1pt)$);
            \draw [-left to] ($(dots.west) + (0pt, -1pt)$) -- ($(X1.east) + (0pt, -1pt)$);

            \draw [-left to] ($(dots.east) + (0pt, 1pt)$) -- ($(Xn.west) + (0pt, 1pt)$);
            \draw [-left to] ($(Xn.west) + (0pt, -1pt)$) -- ($(dots.east) + (0pt, -1pt)$);
\end{tikzpicture}\end{center} has a positive steady state parametrization given by 
\[\begin{cases}
    x_i = \left(\dfrac{k_i'\cdots k_{n-1}'}{k_{i}\cdots k_{n-1}}\right)^{1/f_i}\tau^{f_n/f_i},&i = 0,1,\ldots, n-1\\
    x_n = \tau,&\tau\in\mathbb R_{>0}
\end{cases}\] where $f_i$ denotes the kinetic order value associated to the species $X_i, k_i$ denotes the rate constant for the forward reaction $X_i\to X_{i+1},$ and $k_{i}'$ denotes the rate constant for the backward reaction $X_{i+1}\to X_{i}$ for all $i = 0,1,\ldots, n-1.$
\end{theorem}


\begin{proof} First, we decompose the system $\NN^n$ into $n$ identical and independent subsystems $(\NN_i^n, \KK_i^n)$ (or simply $\NN_i^n$) whose reaction sets are induced by the reversible reaction from $X_{i}$ to $X_{i+1}$ for $i = 0,1,\ldots, n-1.$ 

Next, we derive the positive steady state parametrization for the system when $n = 1.$ In this case, we have the simplest case of the network which has no nontrivial independent decomposition. Then, the SFRF of the system is given by 
\[f(x_0, x_1) = k_0x_0^{f_0}\begin{bmatrix}
    -1\\1
\end{bmatrix} + k_0'x_1^{f_1}\begin{bmatrix}
    1\\-1
\end{bmatrix} = \begin{bmatrix}
    -k_0x_0^{f_0} + k_0'x_1^{f_1}\\
    \phantom{-}k_0x_0^{f_0} - k_0'x_1^{f_1}
\end{bmatrix}\] where $k_0$ and $k_0'$ are the rate constants for the forward and backward reactions, respectively. Since the rows are linearly dependent, we can admit infinitely many solutions to the system. If we set the free parameter $\tau$ to be the concentration $x_1$ at steady state, then the positive steady states of the system $\NN^1$ are given by
\[\begin{cases}
    x_0 = \left(\dfrac{k_0'}{k_0}\right)^{1/f_0}\tau^{f_1/f_0},\\
    x_1 = \tau,
\end{cases}\quad \tau\in\mathbb R_{>0}.\]

In general, we have for any subnetwork $\NN_\beta^n$ involving, say $X_\beta$ and $X_{\beta+1}$ for $\beta = 0,1,\ldots, n-1$ that its positive steady states are given by
\[\begin{cases}
    x_{\beta} = \left(\dfrac{k_\beta'}{k_{\beta}}\right)^{1/f_{\beta}}\tau^{f_{\beta+1}/f_{\beta}},\\
    x_{\beta+1} = \tau,
\end{cases}\quad\tau\in\mathbb R_{>0}.\]

We proceed to the case when $n = 2.$ Here, we have the system $\NN^2$ whose associated CRN is given by
\begin{center}\begin{tikzpicture}[baseline=(current  bounding  box.center)]
			\node (X0) {$X_0$};
            \node (X1) [right of=X0] {$X_1$};
            \node (X2)[right of=X1]{$X_2$};

            \draw [-left to] ($(X0.east) + (0pt, 1pt)$) -- ($(X1.west) + (0pt, 1pt)$);
            \draw [-left to] ($(X1.west) + (0pt, -1pt)$) -- ($(X0.east) + (0pt, -1pt)$);

            \draw [-left to] ($(X1.east) + (0pt, 1pt)$) -- ($(X2.west) + (0pt, 1pt)$);
            \draw [-left to] ($(X2.west) + (0pt, -1pt)$) -- ($(X1.east) + (0pt, -1pt)$);
\end{tikzpicture}\end{center} which can be decomposed into two independent and identical subnetworks $\NN_1^2$ and $\NN_2^2.$ Thus, we get the positive steady states of $\NN_1^2$ to be
\[\begin{cases}
    x_0 = \left(\dfrac{k_0'}{k_0}\right)^{1/f_0}\tau_1^{f_1/f_0}\\
    x_1 = \tau_1,
\end{cases}\quad \tau_1\in\mathbb R_{>0}\] and similarly for $\NN_2^2,$ we have
\[\begin{cases}
    x_1 = \left(\dfrac{k_1'}{k_1}\right)^{1/f_1}\tau_2^{f_2/f_1}\\
    x_2 = \tau_2,
\end{cases}\quad\tau_2\in\mathbb R_{>0}.\] By \autoref{thm:main-result}, we can take the intersection of the sets of positive steady states of each independent subnetwork to get the set of positive steady states for the entire network. Here, since we get \[\tau_1 = \left(\frac{k_1'}{k_1}\right)^{1/f_1}\tau_2^{f_2/f_1},\] we have the concentration of $x_0$ at steady state to be \[x_0 = \left(\frac{k_0'}{k_0}\right)^{1/f_0}\left[\left(\frac{k_1}{k_1'}\right)^{1/f_1}\tau_2^{f_2/f_1}\right]^{f_1/f_0} = \left(\frac{k_0'k_1'}{k_0k_1}\right)^{1/f_0}\tau_2^{f_2/f_0}.\] Therefore, the positive steady states of $\NN^2$ are given by \[\begin{cases}
    x_0 = \left(\dfrac{k_0'k_1'}{k_0k_1}\right)^{1/f_0}\tau_2^{f_2/f_0},\\
    x_1 = \left(\dfrac{k_1'}{k_1}\right)^{1/f_1}\tau_2^{f_2/f_1},\\
    x_2 = \tau_2,
\end{cases}\quad\tau_2 \in\mathbb R_{>0}.\]

To complete our proof, we proceed by induction with the base case $n = 2.$ Indeed, we have shown that this is true from the above. For our inductive hypothesis, we assume that for every $n = m$ with $m \geq 2,$ the positive steady state solution of the system is given by 
\[\begin{cases}
    x_i = \left(\dfrac{k_i'\cdots k_{n-1}'}{k_i\cdots k_{n-1}}\right)^{1/f_i}\sigma^{f_n/f_i},&i = 0,1,\ldots, n-1\\
    x_n = \sigma,&\sigma\in\mathbb R_{>0}
\end{cases}.\] Denote this set of positive steady states by $E_+(\NN^n).$ We have shown that $E_+(\NN^n)$ satisfies the conditions {of stoichiometric independence, as well as of $\widehat T$-independence.} As such, by \autoref{thm:main-result}, we can write \[E_+\left(\NN^{n+1}\right) = E_+\left(\NN^{n}\right)\cap E_+\left(\NN^{n+1}_{n+1}\right)\] where $\NN^{n+1}_{n+1}$ corresponds to the system involving the reversible reaction from $X_n$ to $X_{n+1}.$ For $\NN^{n+1}_{n+1},$ the positive steady state solution is given by \[\begin{cases}
    x_{n} = \left(\dfrac{k_n'}{k_n}\right)^{1/f_n}\tau^{f_{n+1}/f_n},\\
    x_{n+1} = \tau,
\end{cases}\quad\tau\in\mathbb R_{>0}.\]
Taking the intersection of $E_+(\NN^n)$ and $E_+(\NN^{n+1}_{n+1})$ gives \[\sigma = \left(\dfrac{k_n'}{k_n}\right)^{1/f_n}\tau^{f_{n+1}/f_n}\] so that at the steady state of $\NN^{n+1},$ we have 
\begin{align*} x_i &= \left(\dfrac{k_i'\cdots k_{n-1}'}{k_i\cdots k_{n-1}}\right)^{1/f_i}\cdot\left[\left(\dfrac{k_n'}{k_n}\right)^{1/f_n}\tau^{f_{n+1}/f_n}\right]^{f_{n}/f_i}\\
&= \left(\frac{k_i'\cdots k_{n-1}'k_n'}{k_i\cdots k_{n-1}k_n}\right)^{1/f_i}\tau^{f_{n+1}/f_n}
\end{align*}
for each $i = 0,1,\ldots, n-1.$ Therefore, the parameterized positive steady states of $\NN^{n+1}$ is given by \[
\begin{cases}
    x_i = \left(\dfrac{k_i'\cdots k_n'}{k_i\cdots k_n}\right)^{1/f_i}\tau^{f_{n+1}/f_i},&i = 0,1,\ldots, n\\
    x_{n+1} = \tau,&\tau\in\mathbb R_{>0}
\end{cases}\quad\tau\in\mathbb R_{>0},
\] which completes the inductive step, and in whole, our proof.
\end{proof}

In mass actions systems, the kinetic order values $f_i$ equals $1$ for all $i = 0,1,\ldots, n-1.$ This gives us the following corollary of our result.
\begin{corollary}
    If the kinetics in the system in \autoref{thm:result-2} is mass action, then the positive steady state solution of the system is given by \[\begin{cases}
        x_i = \dfrac{k_i'\cdots k_{n}'}{k_i\cdots k_n}\cdot\tau\\
        x_n = \tau,
    \end{cases}\quad \tau\in\mathbb R_{>0}.\]
\end{corollary}

For systems whose underlying follows the form of $\NN^n$ in \eqref{eqn:futile-cyc}, our result gives a convenient method of determining the positive steady state solution of the system through its rate constants and kinetic orders. 

\section{Summary, Conclusion, and Outlook}
\label{summary}

In this work, we consider {power law} systems, which include the mass action, where the underlying CRN can be decomposed into independent subnetworks.
For such systems, we have shown that if 
the ranks of the augmented matrices of kinetic order vectors of the subnetworks sum up to the rank of the augmented matrix of kinetic order vectors of the entire network, then the set of the positive steady states of the given power law system is nonempty if and only if each set of positive steady states of its corresponding subsystems is nonempty.
We have illustrated the result using a carbon cycle network, which satisfies the assumptions and another network, which does not satisfy the assumptions of the theorem. Lastly, we have provided a method to derive steady state parametrizations of PL-RDK systems with identical stoichiometrically independent subnetworks. Thus, our results can be used to analyze positive steady state-related properties for such systems.

{
Finally, it would also be intriguing to investigate whether there exist connections between Feinberg's notion of independent subnetworks and inheritance of steady states, in particular, the inheritance of oscillations \cite{Banaji:oscillation}, local bifurcations \cite{Banaji:bifurcations}, and nondegenerate multistationarity \cite{banaji:nondegenerate:multistationarity}.}

\section*{Acknowledgement}
AJLJA acknowledges the Merit undergraduate scholarship under the Department of Science and Technology - Science Education Institute (DOST-SEI), Philippines.



\end{document}